\documentclass[11pt,a4paper]{amsart}

\usepackage{wasowicz}
\usepackage{times}

\renewcommand{\le}{\leqslant}
\renewcommand{\leq}{\leqslant}

\renewcommand{\geq}{\geqslant}

\newcommand{\barycenter}{\mathbf{b}}

\info{Submitted to Opuscula Math.}

\begin{document}

\title[]
      {Hermite--Hadamard type inequalities for Wright-convex functions of several variables}
\author[]{Dorota Śliwińska \and \SW}
\address{\SWaddr}
\email[\SW]{\SWmail}
\email[Dorota Śliwińska]{\color{blue}dsliwinska@ath.bielsko.pl}


\date{\today}

\begin{abstract}
We present Hermite--Hadamard type inequalities for Wright-con\-vex, strongly convex and strongly Wright-convex functions of several variables defined on simplices.
\end{abstract}

\maketitle

\section{Introduction}
One of the most classical inequalities in the theory od convex functions is the Hermite--Hadamard inequality. It states that if $f:[a,b] \to \mathbb{R}$ is convex then
\[f \left( \frac{a+b}{2} \right) \leq  \frac{1}{b-a}\int_a^bf(x)\dx \leq \frac{f(a)+f(b)}{2}\]
 It plays an important role in convex analysis, so in the literature one can find its various generalizations and applications. For example, an exhausting study of this inequality is given in~the book~\cite{DraPea02}.

 Recall that a function $f:D\to\R$, where $D\subset\R^n$ is a~convex set, is called \emph{Wright-convex} (W-convex for short), if
 \[
  f\bigl(tx+(1-t)y\bigr)+f\bigl((1-t)x+ty\bigr)\le f(x)+f(y)
 \]
 for any $x,y\in D$ and $t\in[0,1]$. Trivially we can see that any convex function is necessarily W-convex and any W-convex function is Jensen-convex (i.e. it fulfills the above inequality with $t=\frac{1}{2}$). However, these inclusions are proper. It is evident, if one knows the famous Ng's representation (cf.~\cite{Ng87}). It states that any W-convex function defined on an~open and convex set $D\subset\R^n$ is the sum of an additive function $a:\R^n\to\R$ and a convex function $g:D\to\R$. Therefore, if $f:\R\to\R$ is W-convex, then either $f$ is continuous (and then convex), or the graph of~$f$ is a dense subset of a~plane. Hence by putting $f(x)=|a(x)|$, where $a:\R\to\R$ is a~discontinuous additive function, we obtain a~Jensen-convex function, which is not W-convex. Of course, the function $a$ is a~W-convex function, which is not convex.

 It was natural to generalize the Hermite--Hadamard inequality to the functions of several variables. In the case of simplices, for the first time it was done by Neuman~\cite{Neu90} (see also \cite{Bes08}, \cite{GueSch03} and \cite{Was08MIA} for the functions defined on simplices and \cite{DraPea02}, \cite{NicPer06} for more general domains). Recently Olbryś~\cite{Olb} obtained the following inequality of Hermite--Hadamard type: if $f:\I\to\R$ (where $\I\subset\R$ is an open interval) is W-convex, then
 \begin{equation}\label{eq:Olb}
  2f\Bigl(\frac{a+b}{2}\Bigr)\le\frac{1}{b-a}\int_a^b\bigl(f(x)+f(a+b-x)\bigr)\,\text{d}x\le f(a)+f(b)
 \end{equation}
 for any $a,b\in\I$. Because the note~\cite{Olb} is actually unpublished, let us mention that~\eqref{eq:Olb} is an immediate consequence of the Hermite--Hadamard inequality. It is enough to apply it to the convex (due to Ng's representation of~$f$) function $[a,b]\ni x\mapsto f(x)+f(a+b-x)$.
\par\medskip
 Motivated by this beautiful Olbryś's result we present in this paper its multivariate counterparts. We also give some related inequalities for strongly convex and strongly W-convex functions of several variables.

\section{Definitions and basic properties}
Let $v_0, \dots, v_n \in \mathbb{R}^n$ be affine independent and let  $S=\conv\{v_0,\dots,v_n\}$ be a simplex with vertices $v_0, \dots, v_n.$ Denote by $|S|$ its volume and by $\barycenter$ its barycenter, \ie
\[
 \barycenter=\frac{1}{n}\sum_{i=1}^n v_i\,.
\]
Any element $x \in S$ is uniquely represented by a convex combination of the vertices:
\[
 x=\sum_{i=0}^n t_iv_i\,,
\]
where the coefficients $t_i\geq 0$, $i=0,\dots,n$, with $t_0+\dots+t_n=1$, are called the \emph{barycentric coordinates} of~$x$. Moreover, any $x\in\R^n$ has the above (unique) representation with real scalars summing up to~$1$.

Denote by $C$ the set of all cyclic permutations of $\{0,\dots,n\}$. Any $\sigma\in C$ generates an affine transformation $\sigma:\R^n\to\R^n$ in the manner
\[
 \sigma \Bigl( \sum_{i=0}^nt_iv_i \Bigr)=\sum_{i=0}^n t_{\sigma (i)}v_i\,.
\]
From now on we identify $\sigma\in C$ with the affine map $\sigma$ given as above. For  $\sigma\in C$ and for any function $f:S \to \mathbb{R}$ we define the function $f_{\sigma}:S\to S$ by
\[
f_{\sigma}(x)=f\bigl(\sigma (x)\bigr).
\]
Next we introduce the \emph{symmetrization}~$F$ of a~function~$f$ as follows:
\begin{equation}\label{eq:sym}
 F(x)=\sum_{\sigma \in C} f_{\sigma}(x)\,,\quad x\in S.
\end{equation}
It is easy to observe that $F$ is symmetric with respect to the barycenter, which means that $F\bigl(\sigma(x)\bigr)=F(x)$ for any $\sigma\in C$.

In our article we use the Hermite--Hadamard inequality on simplices, which was firstly given by Neuman~\cite{Neu90}, then reproved by Guessab and Schmeisser~\cite{GueSch03}, Bessenyei~\cite{Bes08} and the second author~\cite[Corollary~3]{Was08MIA}:

\begin{thm}\label{tw:1}
If $f: S \to \mathbb{R}$ is convex then
\begin{equation}\label{eq:sw}
f(\barycenter) \leq \frac{1}{|S|} \int_S f(x)\dx \leq \frac{1}{n+1} \sum_{i=0}^n f(v_i).
\end{equation}
\end{thm}
To prove the Hermite-Hadamard type inequality for W-convex functions, we need two lemmas. The first of them could be found in \cite[Lemma 2.2]{WasWit12}.
\begin{lem}
If $g: S \to R$ is convex then so is $g_{\sigma}$. \label{tw:l1}
\end{lem}
\begin{lem}
If $a: \mathbb{R}^n \to \mathbb{R}$ is additive, then its symmetrization $A=\displaystyle\sum_{\sigma \in C} a_\sigma$ is constant. \label{tw:l2}
\end{lem}
\begin{proof} Any $x\in\R^n$ may be written as $x=\displaystyle\sum_{i=0}^nt_iv_i$ with real scalars $t_0,\dots,t_n$ (possibly not all positive) summing up to~$1$. Then by additivity of~$a$ and affinity of~$\sigma$ we get
\begin{multline*}
 A(x)=\sum_{\sigma \in C} a_{\sigma}(x)=\sum_{\sigma\in C}a\bigl(\sigma(x)\bigr)
 =a\biggl(\sum_{\sigma\in C}\sigma \Bigl(\sum_{i=0}^n t_iv_i\Bigr)\biggr)\\
  =a\biggl(\sum_{\sigma\in C}\Bigl(\sum_{i=0}^n t_i \sigma (v_i)\Bigr)\biggr)
  =a\Bigl(\sum_{i=0}^n v_i\Bigr)\,,
\end{multline*}
which is a constant.
\end{proof}

\section{Hermite--Hadamard type inequality for W-convex functions}

\begin{thm}\label{tw:3}
 Let $D\subset\R^n$ be an open and convex set and $S\subset D$ be a~simplex. If $f:D\to\R$  is W-convex, then its symmetrization~$F$ is convex on~$S$.
\end{thm}
\begin{proof}
Because $f$ is W-convex, then $f=a+g$ for some additive function $a: \mathbb{R}^n \to \mathbb{R}$ and a convex function $g:D\to\R$. The function $A$ (symmetrization of~$a$, \cf~\eqref{eq:sym}) is constant by Lemma~\ref{tw:l2}, while the function~$G$ (symmetrization of~$g$ on~$S$), is convex by~Lemma~\ref{tw:l1}. Thus~$F$ is convex on~$S$.
\end{proof}

\begin{thm}\label{th:HH}
 Let $D\subset\R^n$ be an open and convex set and $S\subset D$ be a~simplex with vertices $v_0,\dots,v_n$ and barycenter $\barycenter$. If $f:D\to\R$ is W-convex, then
\[ (n+1)f(\barycenter) \leq \frac{1}{|S|} \int_S \Bigl( \sum_{\sigma \in C}f_{\sigma}(x) \Bigr) \dx \leq \sum_{i=0}^n f(v_i).  \] \label{tw:eq1}
\end{thm}
\begin{proof}
Let $F=\sum_{\sigma\in C}f_{\sigma}$ be the symmetrization of~$f$ on~$S$. By the previous theorem~$F$ is convex on~$S$. Using the Hermite--Hadamard inequality (\cf~Theorem~\ref{tw:1}) we arrive at
\[
 F(\barycenter) \leq \frac{1}{|S|} \int_S F(x)\dx \leq \frac{1}{n+1} \sum_{i=0}^n F(v_i)\,.
\]
Observe that $\sigma (\barycenter)=\barycenter.$ Therefore
\begin{equation}
F(\barycenter)=\sum_{\sigma \in C}  f_{\sigma}(\barycenter)=\sum_{\sigma \in C} f \left( \sigma (\barycenter) \right)=(n+1)f(\barycenter)\,, \label{eq:uw13}
\end{equation}
while
\begin{multline}
 \sum_{i=0}^n F(v_i)=\sum_{i=0}^n \sum_{\sigma \in C} f_{\sigma}(v_i)=\sum_{i=0}^n \left( f(v_0)+ \dots +f(v_n) \right) \\ = (n+1)\sum_{i=0}^n f(v_i) \label{eq:uw131}
\end{multline}
and the proof is finished.
\end{proof}
\begin{rem}
 For $n=1$ and $S=[a,b]$ we obtain immediately the result due to Olbryś~\cite{Olb} given by~\eqref{eq:Olb}.
\end{rem}

Observe that in the Hermite--Hadamard inequality~\eqref{eq:sw} one fact was essential: the integral mean value $\mathcal{T}[f]=\frac{1}{|S|} \int_S f(x)\dx$ is a positive linear operator. This motivated the second author to prove the operator version of the inequality~\eqref{eq:sw} (\cf~\cite[Theorem 2]{Was08MIA}). Our next theorem offers an extension of Theorem \ref{tw:eq1} to positive linear operators going in this direction.

\begin{thm}
 Let $D\subset\R^n$ be an open and convex set and $S\subset D$ be a~simplex with vertices $v_0,\dots,v_n$ and barycenter $\barycenter$. Let $\T$ be positive linear functional defined (at least) on a~linear subspace of all functions mapping $S$ into $\R$ generated by a cone of convex functions.  Assume that
\[
 \mathcal{T} (\pi_i)=\frac{1}{|S|} \int_S \pi_i(x)\dx, \ i=1, \dots, n\,,
\]
where $\pi_i$ is the projection onto the $i$--th axis and $\mathcal{T} (\mathbf{1})=1$.
If $f:D\to\R$ is W-convex and $F$ is the symmetrization of~$f$ on~$S$, then
\begin{equation}
 F(\barycenter) \leq  \mathcal{T}[F] \leq \frac{1}{n+1} \sum_{i=0}^n F(v_i)\,. \label{eq:th7}
\end{equation}
\end{thm}
\begin{proof}
Take an arbitrary affine function $\varphi:\R^n\to\mathbb{R}$. It has a~form
\[
 \varphi(x)=\sum_{i=0}^n \alpha_ix_i+\beta =\sum_{i=0}^n \alpha_i\pi_i(x)+\beta\,,
 \quad x=(x_1,\dots,x_n)\in\R^n
\]
for some scalars $\alpha_0,\dots,\alpha_n,\beta$.
The linearity yields
\begin{align*}
 \T[\varphi]&=\T\Bigl[\sum_{i=0}^n \alpha_i\pi_i+\beta\Bigr]=\sum_{i=0}^n\alpha_i\T[
  \pi_i]+\beta\mathcal{T}[\mathbf{1}]
 =\sum_{i=0}^n \frac{\alpha_i}{|S|}\int_S \pi_i(x)\dx +\beta\\
&  =\frac{1}{|S|}\int_S\Bigl(\sum_{i=0}^n\alpha_i\pi_i(x)+\beta\Bigr)\dx
  =\frac{1}{|S|}\int_S\varphi(x)\dx\,.
\end{align*}
 Therefore $\T$ meets the assumptions of \cite[Theorem~2]{Was08MIA}. Hence, by convexity of~$F$, the inequality~\eqref{eq:th7} holds.
\end{proof}
Of course, taking in the above theorem $\T[f]=\frac{1}{|S|}\int_Sf(x)\dx$, we obtain immediately the Theorem~\ref{th:HH}.

\section{Hermite--Hadamard type inequality for strongly convex functions}

Let $D\subset\R^n$ be a convex set and $c>0$. The function $f:D\to\mathbb{R}$ is called \emph{strongly convex with modulus}~$c$, if
\[
 f\bigl(tx+(1-t)y\bigr)\le tf(x)+(1-t)f(y)-ct(1-t)\|x-y\|^2
\]
for all $x,y \in D$ and $t \in [0,1].$ Strongly convex functions were introduced by Polyak~\cite{Pol66} (see also \cite{MerNik10} for some interesting remarks on this class of functions). Let us only mention that a~strongly convex function is necessarily convex, but the converse does not hold (for instance, affine functions are not strongly convex).

Below we present the multivariate counterpart of a~result due to Merentes and Nikodem~\cite{MerNik10}.
\begin{thm}\label{tw:sc}
If $f: S \to \mathbb{R}$ is strongly convex with modulus $c$, then
\begin{multline*}
 f(\barycenter)+c\biggl(\frac{1}{|S|}\int_S\|x\|^2\dx-\|\barycenter\|^2\biggr)
 \leq\frac{1}{|S|}\int_Sf(x)\dx\\
 \leq\frac{1}{n+1}\sum_{i=0}^nf(v_i)
  +c\biggl(\frac{1}{|S|}\int_S\|x\|^2\dx-\frac{1}{n+1}\sum_{i=0}^n\|v_i\|^2\biggr) .
\end{multline*}
\end{thm}
\begin{proof}
We take a function $g:S \to \mathbb{R}$ of the form $g=f-c\|\cdot\|^2.$ Since $f$ is strongly convex with modulus $c$, then $g$ is convex (for a~quick reference see~\cite{HirLam01} or~\cite{MerNik10}). Therefore $g$ satisfies the Hermite--Hadamard inequality \eqref{eq:sw}:
\[g(\barycenter) \leq \frac{1}{|S|} \int_S g(x)\dx \leq \frac{1}{n+1} \sum_{i=0}^n g(v_i).\]
Then we arrive at
\begin{multline*}
 f(\barycenter)-c\|\barycenter\|^2\leq \frac{1}{|S|} \int_S f(x)\dx-\frac{c}{|S|}\int_S\|x\|^2\dx\\
 \leq \frac{1}{n+1} \sum_{i=0}^n f(v_i)-\frac{c}{n+1} \sum_{i=0}^n\|v_i\|^2\,.
\end{multline*}
and our result follows by adding the term $\frac{c}{|S|}\int_S\|x\|^2\dx$ to both sides of the above inequality.
\end{proof}

Denote by $S_1$ the unit simplex in $\R^n$, i.e. the simplex with vertices $e_0=(0,0,\dots,0)$, $e_1=(1,0,\dots,0)$, \dots, $e_n=(0,\dots,0,1)$. Then \[\barycenter=\frac{1}{n+1} \sum_{i=0}^ne_i\,.\]
It is well-known that $|S|=\frac{1}{n!}$. The second author noticed in~\cite{Was08MIA} (proof of Corollary~8) that
\begin{equation}\label{eq:c2}
\int_{S_1} \pi_i^2(x)\dx=\frac{2}{(n+2)!}\,,\quad i=1,\dots , n\,.
\end{equation}

For strongly convex functions defined on the unit simplex~$S_1$, Theorem~\ref{tw:sc} together with~\eqref{eq:c2} gives us
\begin{cor}
If $f:S_1 \to  \mathbb{R}$ is strongly convex with modulus $c$, then
\begin{multline*}
 f(\barycenter)+\frac{cn^2}{(n+1)^2(n+2)}\leq n! \int_{S_1} f(x)\dx\\
 \leq \frac{1}{n+1} \sum_{i=0}^nf(e_i)-\frac{cn^2}{(n+1)(n+2)} .
\end{multline*}
\end{cor}
\begin{rem} \label{tw:rem11}
 For $n=1$ we obtain the inequality
\[
 f\biggl(\frac{1}{2} \biggr) +\frac{c}{12} \leq \int_0^1f(x)\dx \leq \frac{f(0)+f(1)}{2} - \frac{c}{6}\,,
\]
which corresponds with the result of Merentes and Nikodem presented in \cite[Theorem 6]{MerNik10} (the full version of their inequality could be derived from~Theorem~\ref{tw:sc} by setting $n=1$ and an arbitrary~compact interval $[a,b]$ in the role of~$S$).
\end{rem}

\section{Hermite--Hadamard type inequality for strongly W-convex functions}

Let $D\subset\R^n$ be a convex set and $c>0$. The function $f:D \to \mathbb{R}$ is called \emph{strongly W-convex with modulus $c$}, if
\[
 f\bigl(tx+(1-t)y\bigr)+f \bigl( (1-t)x+ty \bigr) \leq f(x)+f(y)-2ct(1-t)\|x-y\|^2
\]
for all $x,y \in D$ and $t \in [0,1]$. Such functions were introduced by Merentes, Nikodem and Rivas in~\cite{MerNikRiv11}. We present below a~counterpart of Theorem~\ref{tw:3} for strongly W-convex functions.

\begin{thm}
 Let $D\subset\R^n$ be an open and convex set and $S\subset D$ be a~simplex. If $f:D\to\R$ is strongly W-convex with modulus $c$, then its symmetrization~$F$ is strongly convex on~$S$ with modulus $(n+1)c$. In particular, $F$ is integrable on $S$.
\end{thm}
\begin{proof}
 Since $f$ is strongly W-convex with modulus~$c$, there exists a~W-convex function $h:D\to\mathbb{R}$ such that $f(x)=h(x)+c\|x\|^2$, $x\in D$ (\cf~\cite[Corollary~5]{MerNikRiv11}). Take an arbitrary vector $x\in S$. If $\sigma\in C$, then $\|\sigma(x)\|=\|x\|$, $x\in S$, whence $f_{\sigma}(x)=h_{\sigma}(x)+c\|x\|^2$. Therefore
 \[
  F(x)=\sum_{\sigma \in C} f_{\sigma}(x)=\sum_{\sigma \in C}h_{\sigma}(x) +\sum_{\sigma \in C} c\|x\|^2=\sum_{\sigma \in C}h_{\sigma}(x)+(n+1)c\|x\|^2\,.
 \]
 Theorem~\ref{tw:3} yields that a~function $\displaystyle\sum_{\sigma\in C}h_{\sigma}$ is convex on~$S$. Then $F$ is strongly convex on~$S$ with modulus $(n+1)c$ (\cf~\cite{HirLam01} or~\cite{MerNik10}) and the proof is finished.
\end{proof}

\begin{cor} \label{tw:swc}
 Let $D\subset\R^n$ be an open and convex set and $S\subset D$ be a~simplex with vertices $v_0,\dots,v_n$ and barycenter $\barycenter$. If $f:D\to\R$ is strongly W-convex with modulus~$c$, then
\begin{multline*}
f(\barycenter) +c \biggl( \frac{1}{|S|} \int_S \|x\|^2\dx -\|\barycenter\|^2 \biggr) \leq \frac{1}{(n+1)|S|} \int_S \biggl( \sum_{\sigma \in C}f_{\sigma}(x)\dx \biggr) \\
\leq \frac{1}{n+1} \sum_{i=0}^nf(v_i)+c \biggl( \frac{1}{|S|} \int_S \|x\|^2\dx-\frac{1}{n+1}\sum_{i=0}^n \|v_i\|^2  \biggr)
\end{multline*}
\end{cor}
\begin{proof}
Since the symmetrization $F$ of the function $f$ is strongly convex on~$S$ with modulus $(n+1)c,$ then by virtue of Theorem~\ref{tw:sc}
\begin{multline*}
 F(\barycenter)+(n+1)c\biggl(\frac{1}{|S|}\int_S\|x\|^2\dx-\|\barycenter\|^2\biggr)
 \leq\frac{1}{|S|}\int_SF(x)\dx\\
 \leq\frac{1}{n+1}\sum_{i=0}^nF(v_i)
  +(n+1)c\biggl(\frac{1}{|S|}\int_S\|x\|^2\dx-\frac{1}{n+1}\sum_{i=0}^n\|v_i\|^2\biggr) .
\end{multline*}
We have by \eqref{eq:uw13}, \eqref{eq:uw131}
\[
 F(\barycenter)=(n+1)f(\barycenter)\quad  \text{and} \quad \sum_{i=0}^nF(v_i)=(n+1) \sum_{i=0}^n f(v_i)\,,
\]
from which the Corollary follows.
\end{proof}

For strongly W-convex functions on the unit simplex~$S_1$, Corollary~\ref{tw:swc} together with~\eqref{eq:c2} gives us
\begin{cor}
 Let $D\subset\R^n$ be an open and convex such that $S_1\subset D$. If $f:D\to\R$ is strongly W-convex with modulus $c$, then
\begin{multline*}
f(\barycenter)+\frac{cn^2}{(n+1)^2(n+2)} \leq \frac{n!}{(n+1)} \int_{S_1} \biggl( \sum_{\sigma \in C} f_{\sigma}(x)\dx \biggr)\\
 \leq \frac{1}{n+1} \sum_{i=0}^nf(v_i)-\frac{cn^2}{(n+1)(n+2)} .
\end{multline*}
\end{cor}

\begin{rem}
 For $n=1$ we obtain the inequality
\[
 f\biggl( \frac{1}{2} \biggr) +\frac{c}{12} \leq\frac{1}{2} \int_0^1\bigl( f(x) +f(1-x)\bigr)  \dx \leq \frac{f(0)+f(1)}{2} - \frac{c}{6}\,,
\]
(compare with~Remark \ref{tw:rem11}). The similar inequality for an arbitrary compact interval $[a,b]$ could be derived directly from Corollary~\ref{tw:swc}.
\end{rem}


\begin{thebibliography}{10}

\bibitem{Bes08}
 M.~Bessenyei,
 \emph{The {H}ermite--{H}adamard inequality on simplices},
 Amer. Math. Monthly \textbf{115} (2008), 339--345.

\bibitem{DraPea02}
 S.S. Dragomir, C.E.M. Pearce,
 \emph{Selected topics on {H}ermite--{H}adamard inequalities and applications},
 RGMIA Monographs, Victoria University, 2002.
 \\ONLINE: \url{http://www.rgmia.org/monographs/hermite\_hadamard.html}.

\bibitem{GueSch03}
 A. Guessab, G. Schmeisser,
 \emph{Convexity results and sharp error estimates in approximate multivariate integration},
 Math. Comp. \textbf{73}~(2004), 1365--1384.

\bibitem{HirLam01}
 J.B. Hiriart-Urruty, C. Lemar{\'e}chal,
 \emph{Fundamentals of convex analysis},
 Springer--Verlag, Berlin Heidelberg, 2001.

\bibitem{MerNik10}
 N. Merentes, K. Nikodem,
 \emph{Remarks on strongly convex functions},
 Aequationes Math. \textbf{80}~(2010), 193--199.

\bibitem{MerNikRiv11}
 N. Merentes, K. Nikodem, S. Rivas,
 \emph{Remarks on strongly {W}right-convex functions},
 Ann. Polon. Math. \textbf{102}~(2011), 271--278.

\bibitem{Neu90}
 E. Neuman,
 \emph{Inequalities involving multivariate convex functions II},
 Proc. Amer. Math. Soc. \textbf{109}~(1990), 965--974.

\bibitem{Ng87}
 C.T. Ng.
 Functions generating {S}chur-convex sums.
 In {\em General inequalities, 5 ({O}berwolfach, 1986)}, volume~80 of
  {\em Internat. Schriftenreihe Numer. Math.}, pages 433--438. Birkh\"auser,
  Basel, 1987.

\bibitem{NicPer06}
 C.P. Niculescu,  L.E. Persson,
 \emph{Convex functions and their applications. A contemporary approach},
 Springer, New York 2006.

\bibitem{Olb}
 A. Olbryś,
 \emph{On some inequalities equivalent to the Wright convexity},
 submitted.

\bibitem{Pol66}
 B.~T. Polyak,
 \emph{Existence theorems and convergence of minimizing sequences in extremum problems with restrictions},
 Soviet Math. Dokl. \textbf{7}~(1966), 72--75.

\bibitem{WasWit12}
 Sz. Wąsowicz, A.~Witkowski,
 \emph{On some inequality of Hermite--Hadamard type},
 Opuscula Math.~\textbf{32}~(2012), 591--600.

\bibitem{Was08MIA}
 Sz. Wąsowicz,
 \emph{Hermite--{H}adamard-type inequalities in the approximate integration},
 Math. Inequal. Appl.~\textbf{11}~(2008), 693--700.

\end{thebibliography}
\end{document}